\documentclass[12pt]{article}
\usepackage{amssymb,amsmath,latexsym,amsthm}
\usepackage{mathrsfs,mathabx}
\usepackage{environ}
\usepackage{url}
\usepackage{enumerate}
\usepackage{ifpdf}
\ifpdf
    \usepackage[pdftex]{hyperref}
\else
    \usepackage[dvips]{hyperref}
\fi

\title{%
    Finite Markov chains coupled to general Markov processes and an application
    to metastability I
}
\author{%
    Thomas G. Kurtz\\
    University of Wisconsin-Madison
    \and
    Jason Swanson\thanks{%
        Supported in part by the VIGRE grant of University of Wisconsin-Madison
        and by NSA grant H98230-09-1-0079.
    }\\
    University of Central Florida
}
\date{January 6, 2021}

\usepackage{color}

\begin{document}

\newtheorem{thm}{Theorem}[section]
\newtheorem{cor}[thm]{Corollary}
\newtheorem{prop}[thm]{Proposition}
\newtheorem{lemma}[thm]{Lemma}
\newtheorem{assum}[thm]{Assumption}
\newtheorem{condition}[thm]{Condition}
\theoremstyle{definition}
\newtheorem{defn}[thm]{Definition}
\newtheorem{expl}[thm]{Example}
\theoremstyle{remark}
\newtheorem{rmk}[thm]{Remark}
\numberwithin{equation}{section}

\def\al{\alpha}
\def\be{\beta}
\def\ga{\gamma}
\def\Ga{\Gamma}
\def\de{\delta}
\def\De{\Delta}
\def\ep{\varepsilon}
\def\eps{\varepsilon}
\def\ze{\zeta}
\def\th{\theta}
\def\ka{\kappa}
\def\la{\lambda}
\def\La{\Lambda}
\def\vpi{\varpi}
\def\si{\sigma}
\def\Si{\Sigma}
\def\ph{\varphi}
\def\om{\omega}
\def\Om{\Omega}

\def\wt{\widetilde}
\def\wh{\widehat}
\def\wch{\widecheck}
\def\ol{\overline}
\def\ds{\displaystyle}

\def\nab{\nabla}
\def\pa{\partial}
\def\To{\Rightarrow}
\def\eqd{\overset{d}{=}}
\def\emp{\emptyset}

\def\pf{\noindent{\bf Proof.} }
\def\qed{\hfill $\Box$}

\providecommand{\flr}[1]{\left\lfloor{#1}\right\rfloor}
\providecommand{\ceil}[1]{\left\lceil{#1}\right\rceil}
\providecommand{\ang}[1]{\left\langle{#1}\right\rangle}


\def\bA{\mathbb{A}}
\def\bB{\mathbb{B}}
\def\bC{\mathbb{C}}
\def\bD{\mathbb{D}}
\def\bE{\mathbb{E}}
\def\bF{\mathbb{F}}
\def\bG{\mathbb{G}}
\def\bH{\mathbb{H}}
\def\bI{\mathbb{I}}
\def\bJ{\mathbb{J}}
\def\bK{\mathbb{K}}
\def\bL{\mathbb{L}}
\def\bM{\mathbb{M}}
\def\bN{\mathbb{N}}
\def\bO{\mathbb{O}}
\def\bP{\mathbb{P}}
\def\bQ{\mathbb{Q}}
\def\bR{\mathbb{R}}
\def\bS{\mathbb{S}}
\def\bT{\mathbb{T}}
\def\bU{\mathbb{U}}
\def\bV{\mathbb{V}}
\def\bW{\mathbb{W}}
\def\bX{\mathbb{X}}
\def\bY{\mathbb{Y}}
\def\bZ{\mathbb{Z}}

\def\bfA{{\bf A}}
\def\bfB{{\bf B}}
\def\bfC{{\bf C}}
\def\bfD{{\bf D}}
\def\bfE{{\bf E}}
\def\bfF{{\bf F}}
\def\bfG{{\bf G}}
\def\bfH{{\bf H}}
\def\bfI{{\bf I}}
\def\bfJ{{\bf J}}
\def\bfK{{\bf K}}
\def\bfL{{\bf L}}
\def\bfM{{\bf M}}
\def\bfN{{\bf N}}
\def\bfO{{\bf O}}
\def\bfP{{\bf P}}
\def\bfQ{{\bf Q}}
\def\bfR{{\bf R}}
\def\bfS{{\bf S}}
\def\bfT{{\bf T}}
\def\bfU{{\bf U}}
\def\bfV{{\bf V}}
\def\bfW{{\bf W}}
\def\bfX{{\bf X}}
\def\bfY{{\bf Y}}
\def\bfZ{{\bf Z}}

\def\cA{\mathcal{A}}
\def\cB{\mathcal{B}}
\def\cC{\mathcal{C}}
\def\cD{\mathcal{D}}
\def\cE{\mathcal{E}}
\def\cF{\mathcal{F}}
\def\cG{\mathcal{G}}
\def\cH{\mathcal{H}}
\def\cI{\mathcal{I}}
\def\cJ{\mathcal{J}}
\def\cK{\mathcal{K}}
\def\cL{\mathcal{L}}
\def\cM{\mathcal{M}}
\def\cN{\mathcal{N}}
\def\cO{\mathcal{O}}
\def\cP{\mathcal{P}}
\def\cQ{\mathcal{Q}}
\def\cR{\mathcal{R}}
\def\cS{\mathcal{S}}
\def\cT{\mathcal{T}}
\def\cU{\mathcal{U}}
\def\cV{\mathcal{V}}
\def\cW{\mathcal{W}}
\def\cX{\mathcal{X}}
\def\cY{\mathcal{Y}}
\def\cZ{\mathcal{Z}}

\def\sA{\mathscr{A}}
\def\sB{\mathscr{B}}
\def\sC{\mathscr{C}}
\def\sD{\mathscr{D}}
\def\sE{\mathscr{E}}
\def\sF{\mathscr{F}}
\def\sG{\mathscr{G}}
\def\sH{\mathscr{H}}
\def\sI{\mathscr{I}}
\def\sJ{\mathscr{J}}
\def\sK{\mathscr{K}}
\def\sL{\mathscr{L}}
\def\sM{\mathscr{M}}
\def\sN{\mathscr{N}}
\def\sO{\mathscr{O}}
\def\sP{\mathscr{P}}
\def\sQ{\mathscr{Q}}
\def\sR{\mathscr{R}}
\def\sS{\mathscr{S}}
\def\sT{\mathscr{T}}
\def\sU{\mathscr{U}}
\def\sV{\mathscr{V}}
\def\sW{\mathscr{W}}
\def\sX{\mathscr{X}}
\def\sY{\mathscr{Y}}
\def\sZ{\mathscr{Z}}

\def\fA{\mathfrak{A}}
\def\fB{\mathfrak{B}}
\def\fC{\mathfrak{C}}
\def\fD{\mathfrak{D}}
\def\fE{\mathfrak{E}}
\def\fF{\mathfrak{F}}
\def\fG{\mathfrak{G}}
\def\fH{\mathfrak{H}}
\def\fI{\mathfrak{I}}
\def\fJ{\mathfrak{J}}
\def\fK{\mathfrak{K}}
\def\fL{\mathfrak{L}}
\def\fM{\mathfrak{M}}
\def\fN{\mathfrak{N}}
\def\fO{\mathfrak{O}}
\def\fP{\mathfrak{P}}
\def\fQ{\mathfrak{Q}}
\def\fR{\mathfrak{R}}
\def\fS{\mathfrak{S}}
\def\fT{\mathfrak{T}}
\def\fU{\mathfrak{U}}
\def\fV{\mathfrak{V}}
\def\fW{\mathfrak{W}}
\def\fX{\mathfrak{X}}
\def\fY{\mathfrak{Y}}
\def\fZ{\mathfrak{Z}}

\maketitle

\begin{abstract}
    We consider a diffusion given by a small noise perturbation of a dynamical
    system driven by a potential function with a finite number of local minima.
    The  classical results of Freidlin and Wentzell show that the time this
    diffusion spends in the domain of attraction of one of these local minima is
    approximately exponentially distributed and hence the diffusion should
    behave approximately like a Markov chain on the local minima. By the work of
    Bovier and collaborators, the local minima can be associated with the small
    eigenvalues of the diffusion generator. Applying a Markov mapping theorem,
    we use the eigenfunctions of the generator to couple this diffusion to a
    Markov chain whose generator has eigenvalues equal to the eigenvalues of the
    diffusion generator that are associated with the local minima and establish
    explicit formulas for conditional probabilities associated with this
    coupling. The fundamental question then becomes to relate the coupled Markov
    chain to the approximate Markov chain suggested by the results of Freidlin
    and Wentzel. In Part II of this work, we provide a complete analysis of this
    relationship in the special case of a double-well potential in one
    dimension. More generally, the coupling can be constructed for a general
    class of Markov processes and any finite set of eigenvalues of the
    generator.

    \bigskip


\end{abstract}

\section{Introduction}\label{S:intro}

Fix $\ep>0$ and consider the stochastic process,
\begin{equation}\label{SDE}
    X_{\ep}(t) = X_{\ep}(0)
        - \int_0^t \nab F(X_{\ep}(s)) \,ds + \sqrt {2\ep}\,W(t),
\end{equation}
where $F\in C^3(\bR^d)$ and $W$ is a standard $d$-dimensional Brownian motion.
For the precise assumptions on $F$, see Section \ref{S:assumpF}. Let $\ph$ be
the solution to the differential equation $\ph'=-\nab F(\ph)$. We will use
$\varphi_x$ to denote the solution with  $\ph_x(0)=x$. The process $X_{\ep}$ is
a small-noise perturbation of the deterministic process $\ph$.

Suppose that $\cM=\{x_0,\ldots ,x_m\}$ is the set of local minima of the
potential function $F$. The points $x_j$ are stable points for the process
$\ph$. For $X_{\ep}$, however, they are not stable. The process $X_{\ep}$ will
initially gravitate toward one of the $x_j$ and move about randomly in a small
neighborhood of this point. But after an exponential amount of time, a large
fluctuation of the noise term will move the process $X_{\ep}$ out of the domain
of attraction of $x_j$ and into the domain of attraction of one of the other
minima. We say that each point $x_j$ is a point of \textit{metastability} for
the process $X_{\ep}$.

If $X$ is a cadlag process in a complete, separable metric space $S$ adapted to
a right continuous filtration (assumptions that are immediately satisfied for
all processes considered here) and $H$ is either open or closed, then
$\tau^X_H=\inf\{t>0:X(t)\text{ or }X(t-)\in H\}$ is a stopping time (see, for
example, \cite[Proposition 1.5]{EK}). If $x\in S$, let
$\tau^X_x=\tau^X_{\{x\}}$.  We may sometimes also write $\tau^X(H)$, and if the
process is understood, we may omit the superscript.

Let
\begin{equation}\label{domatt}
    D_j = \{x \in \bR^d: \lim_{t\to\infty}\ph_x(t) = x_j\}
\end{equation}
be the domains of attraction of the local minima. It is well-known (see, for
example, \cite{FW}, \cite[Theorem 3.2]{BEGK}, \cite[Theorems 1.2 and 1.4]{BGK},
and \cite{E}) that as $\ep \to 0$, $\tau^{X_{\ep}}(D_j^c)$ is asymptotically
exponentially distributed under $P^{x_j}$.  It is therefore common to
approximate the process $X_{\ep}$ by a continuous time Markov chain on the set
$\cM$ (or equivalently on $\{0,\ldots ,m\}$). In fact, metastability can be
defined in terms of convergence, in an appropriate sense, to a continuous time
Markov chain. (See the survey article \cite{Landim2019} for details.) Beltr\'
{a}n and Landim \cite{Beltran2010,Beltran2012} introduced a general method for
proving the metastability of a Markov chain. Along similar lines, Rezakhanlou
and Seo \cite{Rezakhanlou2018} developed such a method for diffusions. For an
alternative approach using intertwining relations, see \cite{Avena2017}.

In this project, for each $\ep>0$, we wish to capture this approximate Markov
chain behavior by coupling $X_{\ep}$ to a continuous time Markov chain,
$Y_{\ep}$, on $\{0,\ldots ,m\}$. We will refer to the indexed collection of
coupled processes, $\{(X_{\ep},Y_{\ep}):\ep>0\}$, as a \textit{coupling
sequence}. Our objective is to investigate the possibility of constructing a
coupling sequence which satisfies
\begin{equation}\label{tracking-req}
    P(X_{\ep}(t) \in D_j \mid Y_{\ep}(t) = j) \to 1
\end{equation}
as $\ep\to 0$, for all $j$. We also want the transition rate for $Y_{\ep}$ to
go from $i$ to $j$ to be asymptotically equivalent as $\ep\to 0$ to the
transition rate for $X_{\ep}$ to go from a neighborhood of $x_i$ to a
neighborhood of $x_j$. That is, we would like
\begin{equation}\label{rate-match-req}
    E^i[\tau_j^{Y_\ep}] \sim E^{x_i}[\tau^{X_\ep}_{B_{\rho}(x_0)}]
\end{equation}
as $\ep\to0$, for all $i$ and $j$, where $B_{\rho}(x)$ is the ball of radius
$\rho$ centered at $x$.

In this paper (Part I), we develop our general coupling construction. The
construction goes beyond the specific case of interest here. It is a
construction that builds a coupling between a Markov process on a complete and
separable metric space and a continuous-time Markov chain where the generators
of the two processes have common eigenvalues. The coupling is done in such a way
that observations of the chain yield quantifiable conditional probabilities
about the process. This coupling construction is built in Section 
\ref{S:gencoup} and uses the Markov mapping theorem (Theorem \ref{T:MMT}). In
Section \ref{S:coupRd}, we apply this construction method to reversible
diffusions in $\bR^d$ driven by a potential function with a finite number of
local minima.

With this coupling construction in hand, we can build the coupling sequences
described above. In our follow-up work (Part II), we take up the question of the
existence and uniqueness of a coupling sequence that satisfies requirements
\eqref{tracking-req} and \eqref{rate-match-req}.

\section{The general coupling}\label{S:gencoup}

\subsection{Assumptions and definitions}

Given a Markov process $X$ with generator $A$ satisfying Assumption
\ref{A:eigmeas}, we will use the Markov mapping theorem to construct a coupled
pair, $(X,Y)$, in such a way that for a specified class of initial
distributions, $Y$ is a continuous-time Markov chain on a finite state space. The
construction then allows us to explicitly compute the conditional distribution
of $X$ given observations of $Y$.

For explicit definitions of the notation used here and throughout, see Section
\ref{S:MMT}.

\begin{assum}\label{A:eigmeas} 
    Let $E$ be a complete and separable metric space.
    \begin{enumerate}[(i)]
        \item   $A\subset\ol C(E)\times\ol C(E)$.
        \item   $A$ has a stationary distribution $\vpi\in\cP(E)$, which implies
                $\int_EAf\,d\vpi=0$ for all $f\in\cD(A)$.
        \item   For some $m$, there exist signed measures $\vpi_1, \ldots,
                \vpi_m$ on $E$ and positive real numbers $\la_1,\ldots ,\la_m$
                such that, for each $k\in \{1,\ldots,m\}$ and $f\in\cD(A)$,
                \begin{align}
                    &\int_EAf\,d\vpi_k = -\la_k\int_Ef\,d\vpi_k,\label{i}\\
                    &\vpi_k(dx) = \eta_k(x)\vpi(dx),
                        \text{ where }\eta_k \in \ol C(E)\label{ii},\\
                    &\vpi_k(E) = 0.\label{iii}
                \end{align}
    \end{enumerate}
    We define $\vpi_0=\vpi$ and $\eta_0=1$.
\end{assum}

\begin{rmk}
    If $(1,0)\in A$, then \eqref{i} implies \eqref{iii}.
\end{rmk}

\begin{rmk}
    In what follows, we will make use of the assumption that the functions
    $\eta_k$ are continuous.  However, this assumption can be relaxed by
    appealing to the methods in Kurtz and Stockbridge \cite{KSt}.
\end{rmk}

\begin{assum}\label{A:Qmat} 
    Let $E$ be a complete and separable metric space. Let $A\subset\ol
    C(E)\times\ol C(E)$, $m\in\bN$, $Q\in\bR^{(m+1)\times (m+1)}$, and
    $\xi^{(1)},\ldots ,\xi^{(m)}\in\bR^{m+1}$.  
    \begin{enumerate}[(i)]
        \item   $A$ and $m$ satisfy Assumption \ref{A:eigmeas}.
        \item   $Q$ is the generator of a continuous-time Markov chain with
                state space $E_0=\{0,1,\ldots,m\}$ and eigenvalues $\{0,-\la_1,
                \ldots ,-\la_m\}$.
        \item   The vectors $\xi^{(1)},\ldots ,\xi^{(m)}$ are right eigenvectors
                of $Q$, corresponding to the eigenvalues $-\la_1,\ldots,
                -\la_m$.
        \item   For each $i\in \{0,1,\ldots ,m\}$, the function
                \begin{equation}\label{alphas}
                    \al_i(x) = 1 + \sum_{k=1}^m\xi_i^{(k)}\eta_k(x)
                \end{equation}
                satisfies $\al_i(x)>0$ for all $x\in E$. 
    \end{enumerate}
    We define $\xi^{(0)}=(1,\ldots,1)^T$, so that the function
    $\al:E\to\bR^{m+1}$ is given by $\al=\sum_{k=0}^m\xi^{(k)}\eta_k$.  
\end{assum}

\begin{rmk}\label{R:Qmat} 
    Given $(A,m,Q)$ satisfying (i) and (ii) of Assumption \ref{A:Qmat}, it is
    always possible to choose vectors $\xi^{(1)},\ldots ,\xi^{(m)}$ satisfying
    (iii) and (iv). This follows from the fact that each $\eta_k$ is a bounded
    function.
\end{rmk}

\begin{defn}\label{D:Bgen} 
    Suppose $(A,m,Q,\xi^{(0)},\ldots ,\xi^{(m)})$ satisfies Assumption 
    \ref{A:Qmat}. For $0\le j\ne i\le m$, define
    \begin{equation}\label{qdef}
        q_{ij}(x) = Q_{ij}\frac {\al_j(x)}{\al_i(x)}.
    \end{equation}
    Note that $q_{ij}\in C(E)$.  Let $S=E\times E_0$. Define $B\subset\ol C
    (S)\times C(S)$ by
    \begin{equation}\label{Bgen}
        Bf(x,i) = Af(x,i) + \sum_{j\ne i}q_{ij}(x)(f(x,j) - f(x,i)),
    \end{equation}
    where we take 
    \begin{equation}
        \cD(B) = \{f(x,i) = f_1(x)f_2(i): f_1 \in \cD(A), f_2 \in B(E_0)\}
            \label{domB}
    \end{equation}
    In particular, $Af(x,i)=f_2(i)Af_1(x)$. 

    For each $i\in E_0$, define the measure $\al(i,\cdot)$ on $E$ by
    \begin{equation}\label{almeas}
        \al(i,\Ga) = \int_{\Ga}\al_i(x)\vpi(dx),
    \end{equation}
    for all $\Ga\in\cB(E)$. Note that by \eqref{alphas}, \eqref{iii}, and 
    \eqref{ii}, these are probability measures.
\end{defn}

\subsection{Construction of the coupling}

We are now ready to construct our coupled pair, $(X,Y)$, which will have
generator $B$, to prove, for appropriate initial conditions, that the marginal
process $Y$ is a Markov chain with generator $Q$, and to establish our
conditional probability formulas. We first require two lemmas.

\begin{lemma}\label{L:small1}
    In the setting of Definition \ref{D:Bgen}, let $X$ be a cadlag solution of
    the martingale problem for $A$. Then there exists a cadlag process $Y$ such
    that $(X,Y)$ solves the (local) martingale problem for $B$. If $X$ is
    Markov, then $(X,Y)$ is Markov. If the martingale problem for $A$ is
    well-posed, then the martingale problem for $B$ is well-posed.
\end{lemma}

\begin{rmk}
    We are not requiring the $q_{ij}$ to be bounded, so for the process we
    construct,
    \[
        f(X(t),Y(t)) - f(X(0),Y(0)) - \int_0^t Bf(X(s),Y(s))\,ds
    \]
    may only be a local martingale.
\end{rmk}

\begin{proof}[Proof of Lemma \ref{L:small1}]
    Let $X(t)$ be a cadlag solution to the martingale problem for $A$.  Let
    $\{N_{ij}:i,j\in E_0,i\neq j\}$ be a family of independent unit rate Poisson
    processes, which is independent of $X$. Then the equation
    \begin{equation}
        Y(t) = k + \sum_{i\neq j} (j - i) N_{ij} \left({
                \int_0^t 1_{\{i\}} (Y(s)) q_{ij} (X(s)) \,ds
            }\right)\label{yeq}
    \end{equation}
    has a unique solution, and as in \cite{Ku81}, the process $Z=(X,Y)$ is a
    solution of the (local) martingale problem for $B$. If $X$ is Markov, the
    uniqueness of the solution of \eqref{yeq} ensures that $(X,Y)$ is Markov.
    Similarly, $A$ well-posed implies $B$ is well posed.
\end{proof}

\begin{lemma}
    Let $A$ satisfy Assumption \ref{A:eigmeas}. Taking $\psi (x,i)=1+\sum_{j\ne
    i}q_{ij}(x)\ge 1$, if $A$ satisfies Condition \ref{mapcond}, then $B$
    satisfies Condition \ref{mapcond} with $E$ replaced by $S=E\times E_0$.
\end{lemma}

\begin{proof}
    Since $\cD(A)$ is closed under multiplication, $\cD(B)$ defined in
    \eqref{domB} is closed under multiplication.

    Since we are assuming that $\cR(A)\subset\bar {C}(E)$, for each $f\in
    \cD(B)$, there exists $c_f>0$ such that $|Bf(x,i)|\le c_f\psi (x)$.

    Condition \ref{mapcond}(iii) for $A$ and the separability of $B(E_0)$
    implies Condition \ref{mapcond}(iii) for $B_0$.

    Since $A$ is a pre-generator and $B$ is a perturbation of $A$ by a jump
    operator, $B_0$ is a pre-generator.
\end{proof}

\begin{thm}\label{T:main} 
    Suppose $A$ satisfies Condition \ref{mapcond} and $(A,m,Q,\xi^{(1)},\ldots
    ,\xi^{(m)})$ satisfies Assumption \ref{A:Qmat}. Let $B$ be given by
    \eqref{Bgen} and for $p_i\ge 0$, $\sum_{i=0}^mp_i=1$, define
    \[
        \nu(\Ga \times \{i\}) = p_i\al(i,\Ga), \quad \Ga \in \cB(E), i \in E_0.
    \]
    If $\wt Y$ is a cadlag $E_0$-valued Markov chain with generator $Q$ and
    initial distribution $\{p_i\}$, then there exists a solution $(X,Y)$ of the
    martingale problem for $(B,\nu)$ such that $Y$ and $\wt Y$ have the same
    distribution on $D_{E_0}[0,\infty)$, and
    \begin{equation}\label{main}
        P(X(t) \in \Ga \mid \cF_t^Y) = \al(Y(t),\Ga),
    \end{equation}
    for all $t\ge 0$ and all $\Ga\in\cB(E)$.
\end{thm}

\begin{proof}
    We apply Theorem \ref{T:MMT} to the operator $B\subset\ol C(S)\times C(S)$.  

    Let $\ga:S\to E_0$ be the coordinate projection. Let $\wt\al$ be the
    transition function from $E_0$ into $S$ given by the product measure
    $\wt\al(i,\cdot)=\al(i,\cdot)\otimes\de_i^{E_0}$, where $\al(i,\cdot)$ is
    given by \eqref{almeas}. Then $\wt\al(i,\ga^{-1}(i))=1$ and 
    \[
        \wt\psi(i) \equiv \int_S\psi (z)\wt\al(i,dz)
            = \int_E\psi (x,i)\al_i(x)\vpi(dx)
            = 1 + \sum_{j\ne i}Q_{ij} < \infty ,
    \]
    for each $i\in E_0$. Define
    \[
        C = \left\{
                \left(
                    \int_S f(z)\wt\al(\cdot,dz), \int_SBf(z)\wt\al(\cdot,dz)
                \right): f \in \cD(B)
            \right\} \subset \bR^{m+1} \times \bR^{m+1}.
    \]
    The result follow by Theorem \ref{T:MMT}, if we can show that $Cv=Qv$ for
    every  vector $v\in\cD(C)$. Given $f\in\cD(B)$, let
    \[
        \ol f(i) = \int_S f(z)\wt\al(i,dz)
            = \int_E f(x,i)\al(i,dx)
            = \int_E f(x,i)\al_i(x)\vpi(dx).
    \]
    Note that
    \[
        C\ol f(i) = \int_E Bf(x,i)\al_i(x)\vpi(dx).
    \]
    Since $\la_0=0$, by \eqref{i} and the definition of $q_{ij}(x)$,
    \[
        C\ol f(i) = -\sum_{k=0}^m \xi_i^{(k)}\la_k\int_E f(x,i)\eta_k(dx)
            +\sum_{j\ne i} Q_{ij}\int_E \al_j(x)(f(x,j) - f(x,i))\vpi(dx).
    \]
    By assumption $Q\xi^{(k)}=-\la_k\xi^{(k)}$, so
    $-\xi_i^{(k)}\la_k=\sum_{j=0}^mQ_{ij}\xi_j^{(k)}$ and
    \begin{align*}
        -\sum_{k=0}^m \xi_i^{(k)}\la_k\int_E f(x,i)\eta_k(dx)
            &= \sum_{k=0}^m \sum_{j=0}^m Q_{ij}
            \xi_j^{(k)}\int_E f(x,i)\eta_k(dx)\\
        &= \sum_{j=0}^m Q_{ij}\sum_{k=0}^m \xi_j^{(k)}\int_E f(x,i)\eta_k(dx)\\
        &= \sum_{j=0}^m Q_{ij}\int_E f(x,i)\al_j(x)\vpi(dx).
    \end{align*}
    This gives
    \[
        C\ol f(i) = Q_{ii}\int_E f(x,i)\al_i(x)\vpi(dx)
            + \sum_{j\ne i} Q_{ij}\int_S f(x,j)\al_j(x)\vpi(dx)
            = \sum_{j=0}^m Q_{ij}\ol f(j) = Q\ol f(i).
    \]
    It follows that $\wt Y$ is a solution to the martingale problem for $(C,p)$.  

    By Theorem \ref{T:MMT}(a), there exists a solution $Z=(X,Y)$ of the
    martingale problem for $(B,\nu)$ such that $Y=\ga(Z)$ and $\wt Y$ have the
    same distribution on $D_{E_0}[0,\infty)$. Theorem \ref{T:MMT}(b) implies
    \eqref{main}.  
\end{proof}

\begin{rmk}
    In what follows, we may still write expectations with the notation $E^x$ or
    $E^i$, even when we have a coupled process, $(X,Y)$. The meaning will be
    determined by context, depending on whether the integrand of the expectation
    involves only $X$ or only $Y$.  
\end{rmk}

\section{Reversible diffusions}\label{S:coupRd}

\subsection{Assumptions on the potential function}\label{S:assumpF}

We now consider the special case of our coupling when $X$ is a reversible
diffusion on $\bR^d$ driven by a potential function $F$ and a small white noise
perturbation. We will need to use several results from the literature about the
eigenvalues and eigenfunctions of the generator of $X$. We assume the following
on $F$.

\begin{assum}\label{asmF}
    \begin{enumerate}[(i)]
        \item   $F\in C^3(\bR^d)$ and $\lim_{|x|\rightarrow\infty}F(x)=\infty$.
        \item   $F$ has $m+1\ge 2$ local minima ${\cal M}=\{x_0,\ldots,x_m\}$.
        \item   There exist constants $a_i>0$ and $c_i>0$ such that
                $a_2<2a_1-2$, and  
                \begin{align}
                    c_1|x|^{a_1} - c_2 &\le |\nab F(x)|^2 \le c_3|x|^{a_2}+c_4,
                        \label{superquad1}\\
                    c_1|x|^{a_1} - c_2
                        &\le (|\nab F(x)| - 2\De F(x))^2 \le c_3|x|^{a_2}+c_4.
                        \label{superquad2}
                \end{align}
    \end{enumerate}
\end{assum}

\begin{rmk}
Note that $2<a_1\le a_2$. To see this, observe that \eqref{superquad1} implies
$a_1\le a_2$. Thus, $a_1\le a_2<2a_1-2$, which implies $a_1>2$.
\end{rmk}

\begin{lemma}\label{L:superquad3} 
    Under Assumption \ref{asmF}, there exist constants $\wt c_i>0$ such that 
    \begin{equation}\label{superquad3}
        \wt c_1|x|^{\wt a_1} - \wt c_2 \le |F(x)|
            \le \wt c_3|x|^{\wt a_2} + \wt c_4,
    \end{equation}
    where $\wt a_i = a_i/2+1$.
\end{lemma}

\begin{proof}
    Since  
    \[
        F(x) = F(0) + \int_0^1\nab F(sx)\cdot x\,ds,
    \]
    it follows from \eqref{superquad1} that
    \[
        |F(x)| \le |F(0)| + |x|(c_3|x|^{a_2} + c_4)^{1/2},
    \]
    and the upper bound in \eqref{superquad3} follows immediately.

    Since $F\to\infty$, there exists $C>0$ such that $F(x)>-C$ for all
    $x\in\bR^d$, and since $|\nab F|\to\infty$, there exists $R>0$ such that
    $|\nab F(x)|\ge 1$ whenever $|x|\ge R$.  

    Recall that $\ph_x$ satisfies $\ph'_x=-\nab F(\ph_x)$ and $\ph_x(0)=x$, and
    define
    \[
        T_x = \inf\{t \ge 0: |\ph_x(t)| < R\}.
    \]
    Suppose there exists $x$ such that $T_x=\infty$. Then, for all $t>0$,
    \begin{align*}
        -C < F(\ph_x(t)) &= F(x) + \int_0^t \nab F(\ph_x(s))\cdot\ph_x'(s)\,ds\\
        &= F(x) - \int_0^t |\nab F(\ph_x(s))|^2\,ds\\
        &\le F(x) - t.
    \end{align*}
    Therefore, $F(x)\ge t-C$ for all $t$, a contradiction, and we must have
    $T_x<\infty$ for all $x\in\bR^d$.  

    Let $L=\sup_{|x|\le R}F(x)$. By \eqref{superquad1} and the fact that
    $F\to\infty$, we may choose $R'\ge R$ and $C'>0$ such that $F(x)>L$ and
    $|\nab F(x)|\ge C'|x|^{a_1/2}$ whenever $|x|>R'$.  

    Fix $x\in\bR^d$ with $|x|>2R'$, so that $F(x)>L$. Since $|\ph_x(T_x)|=R$, it
    follows that $F(\ph_x(T_x))\le L$. By the continuity of $\ph_x$, we may
    choose $T'\in (0,T_x]$ such that $F(\ph_x(T'))=L$. We then have
    \begin{align*}
        L &= F(x) + \int_0^{T'} \nab F(\ph_x(t))\cdot\ph_x'(t)\,dt\\
        &= F(x) - \int_0^{T'}|\nab F(\ph_x(t))||\ph_x'(t)|\,dt.
    \end{align*}
    Let $T''=\inf\{t\ge 0:|\ph_x(t)|<|x|/2\}$. Note that $F(\ph_x(T'))=L$
    implies $|\ph_x(T')|\le R'<|x|/2$, and therefore $T''\le T'$. Moreover, for
    all $t<T''$, we have $|\ph_x(t)|\ge |x|/2>R'$, which implies
    \[
        |\nab F(\ph_x(t))| \ge C'|\ph_x(t)|^{a_1/2}
            \ge C'\left({\frac {|x|}2}\right)^{a_1/2}.
    \]
    Thus,
    \[
        L \le F(x)
            - C'\left({\frac {|x|}2}\right)^{a_1/2}\int_0^{T''} |\ph_x'(t)|\,dt.
    \]
    But $\int_0^{T''} |\ph_x'(t)|\,dt$ is the length of $\ph_x$ from $t=0$ to
    $t=T''$, which is bounded below by
    \[
        |\ph_x(T'') - \ph_x(0)| \ge |\ph_x(0)| - |\ph_x(T'')|
            = |x| - \frac {|x|}2
            = \frac {|x|}2.
    \]
    Therefore, for all $|x|>2R'$, we have $F(x)\ge C''|x|^{a_1/2+1}-|L|$, where
    $C''=2^{-a_1/2-1}C'$, and this proves the lower bound in \eqref{superquad3}.
\end{proof}

\subsection{Spectral properties of the generator}

Having established our assumptions on $F$, we now turn our attention to the
diffusion process, $X_\ep$, given by \eqref{SDE}. To simplify notation, we may
sometimes omit the $\ep$. The process $X$ has generator $A=\ep\De-\nab
F\cdot\nab$. To show that $A$ meets the requirements of our coupling from
Section \ref{S:gencoup}, we must prove certain results about its eigenvalues and
eigenfunctions. For this, we begin with some notation, a lemma, and two results
from the literature.

Define $\pi(x)=\pi_\ep(x)=e^{-F(x)/2\ep}$. Let
\begin{equation}
    V = V_\ep := \frac {\De\pi}{\pi}
        = \frac 1{4\ep^2}|\nab F|^2 - \frac 1{2\ep}\De F.\label{pot}
\end{equation}

\begin{lemma}\label{L:Vepquad} 
    Let $V_\ep$ be given by \eqref{pot}, where $F$ satisfies Assumption
    \ref{asmF}. Recall the constants $a_i$ from
    \eqref{superquad1}-\eqref{superquad2}. For all $\ep\in (0,1)$, there exist
    constants $c_{i,\ep}>0$ such that
    \[
        c_{1,\ep}|x|^{a_1} - c_{2,\ep} \le V_\ep(x)
            \le c_{3,\ep}|x|^{a_2} + c_{4,\ep}.
    \]
    In particular, $V_\ep\to\infty$ for all $\ep\in(0,1)$.
\end{lemma}

\begin{proof}
    Fix $\ep\in(0,1)$. By \eqref{superquad1} and \eqref{superquad2}, for $x$
    sufficiently large,
    \[
        c|x|^{a_1} \le (|\nab F(x)| - 2\De F)^2 \le C|x|^{a_2},
    \]
    and
    \[
        c|x|^{a_1} \le |\nab F(x)|^2 \le C|x|^{a_2},
    \]
    for some $0<c\le C<\infty$. Note that
    \[
        4V_1 = |\nab F|^2 - 2\De F
            = (|\nab F| - 2\De F) + (|\nab F|^2 - |\nab F|).
    \]
    Hence, for $x$ sufficiently large, $V_1(x)\le C_1|x|^{a_2}$. Also,
    \[
        V_1(x) \ge \frac 14(c|x|^{a_1} - C_2|x|^{a_2/2}).
    \]
    Since $a_1>a_2/2$, it follows that for $x$ sufficiently large, $V_1(x)\ge\wt
    c|x|^{a_1}$. Therefore, there exist constants $\wt c_i>0$ such that
    \[
        \wt c_1|x|^{a_1} - \wt c_2\le V_1(x) \le \wt c_3|x|^{a_2} + \wt c_4,
    \]
    and
    \[
        \wt c_1|x|^{a_1} - \wt c_2 \le |\nab F(x)|^2
            \le \wt c_3|x|^{a_2} + \wt c_4,
    \]
    for all $x\in\bR^d$. Note that
    \[
        V_\ep = \frac 1{\ep}\left({
                V_1 + \left({\frac {1-\ep}{4\ep}}\right)|\nab F|^2
            }\right),
    \]
    so that
    \[
        \frac 1{\ep}V_1 \le V_\ep
            \le \frac 1{\ep}V_1 + \frac 1{\ep^2}|\nab F|^2.
    \]
    From here, the lemma follows easily.
\end{proof}

The following two theorems are from \cite{Da}. Theorem \ref{T:Davies1} is a
consequence of \cite[Theorem 4.5.4]{Da} and \cite[Lemma 4.2.2]{Da}. Theorem
\ref{T:Davies2} is part of \cite[Theorem 2.1.4]{Da}.

\begin{thm}\label{T:Davies1} 
    Let $H=-\De+W$, where $W$ is continuous with $W\to\infty$. Let $\la$ denote
    the smallest eigenvalue of $H$, and $\psi$ the corresponding eigenfunction,
    normalized so that $\|\psi\|_{L^2(\bR^d)}=1$.  Define $Uf=\psi f$ and $\wt
    H=U^{-1}(H-\la)U$. If
    \[
        \wh c_1|x|^{\wh a_1} - \wh c_2 \le |W(x)|
            \le \wh c_3|x|^{\wh a_2} + \wh c_4,
    \]
    where $\wh a_i>0$, $\wh c_i>0$, and $\wh a_2<2\wh a_1-2$, then $e^{-\wt Ht}$
    is an ultracontractive symmetric Markov semigroup on
    $L^2(\bR^d,\psi(x)^2\,dx)$.  That is, for each $t\ge 0$, the operator
    $e^{-\wt Ht}$ is a bounded operator mapping $L^2(\bR^d,\psi(x)^2\,dx)$ to
    $L^\infty(\bR^d,\psi(x)^2\,dx)$.
\end{thm}

\begin{thm}\label{T:Davies2} 
    Let $e^{-Ht}$ be an ultracontractive symmetric Markov semigroup on
    $L^2(\Om,\mu )$, where $\Om$ is a locally compact, second countable
    Hausdorff space and $\mu$ is a Borel measure on $\Om$. If $\mu
    (\Om)<\infty$, then each eigenfunction of $H$ belongs to $L^{\infty}(\Om,\mu
    )$.  
\end{thm}

This next proposition establishes the spectral properties of $A$ that are needed
to carry out the construction of our coupling.  

\begin{prop}\label{P:opbasics} 
    Fix $\ep>0$. The operator $H=-\De+V_{\ep}$ is a self-adjoint operator on
    $L^2(\bR^d)$ with discrete, nonnegative spectrum $\wh\la_k\uparrow\infty$
    and corresponding orthonormal eigenfunctions $\psi_k$. Each $\psi_k$ is
    locally H\"older continuous. Moreover, $\wh\la_0=0$ is simple and $\psi_0$
    is proportional to $\pi$. We define $\mu$ by $\mu(dx)=\pi(x)^2\,dx$ and
    $\vpi=Z^{-1}\mu$, where $Z=\mu (\bR^d)$.  The operator $\wt H$ given by $\wt
    Hf=\pi^{-1}H(\pi f)$ is a self-adjoint operator on $L^2(\vpi)$ with
    eigenvalues $\wh\la_k$ and orthogonal eigenfunctions $\wh\eta_k=\psi_k/\pi$.
    The functions $\wh\eta_k$ have norm one in $L^2(\mu)$, whereas the functions
    $\eta_k=Z^{1/2}\wh\eta_k$ have norm one in $L^2(\vpi)$.  

    For $f\in C_c^\infty(\bR^d)$, we have $-\ep\wt Hf=\ep\De f-\nab F\cdot\nab
    f$. Hence, if we define $A$ by
    \[
        A = \{(f, -\ep\wt Hf): f \in C_c^\infty(\bR^d)\},
    \]
    then $A$ is the generator for the diffusion process given by \eqref{SDE}.
    For each $x\in\bR^d$, \eqref{SDE} has a unique, global solution for all
    time, so that the process $X$ with $X(0)=x$ is a solution to the martingale
    problem for $(A,\de_x)$. The operator $A$ is graph separable, and $\cD(A)$
    is separating and closed under multiplication. The measure $\vpi$ is a
    stationary distribution for $A$. Moreover,
    \[
        \int Af\,d\vpi_k = -\la_k\int f\,d\vpi_k,
    \]
    where $\vpi_k(dx)=\eta_k(x)\vpi(dx)$ and $\la_k=\ep\wh\la_k$. The signed
    measures $\vpi_k$ satisfy $\vpi_k(\bR^d)=0$, and each $\eta_k$ belongs to
    $\ol C(\bR^d)$, the space of bounded, continuous functions on $\bR^d$.
\end{prop}

\begin{proof}
    Note that $V\to\infty$ by Lemma \ref{L:Vepquad}. Therefore, by \cite[Theorem
    XIII.67]{RS}, we have that $H$ is a self-adjoint operator on $L^2(\bR^d)$
    with compact resolvent. It follows (see \cite[pp.  108--109, 119--120, and
    Proposition 1.4.3]{Da}) that $H$ has a purely discrete spectrum and there
    exists a complete, orthonormal set of eigenfunctions
    $\{\psi_k\}_{k=0}^{\infty}$ with corresponding eigenvalues
    $\wh\la_k\uparrow\infty$. Moreover, $\wh\la_0$ is simple and $\psi_0$ is
    strictly positive.  

    Since $V$ is locally bounded, and $(-\De+V-\wh\la_k)\psi_k=0$, \cite[Theorem
    8.22]{GT} implies that, for each compact $K\subset\bR^d$, $\psi_k$ is
    H\"older continuous on $K$ with exponent $\ga(K)$.  

    Define $U:L^2(\mu)\to L^2(\bR^d)$ by $Uf=\pi f$, so that $\wt H=U^{-1}HU$.
    Since $U$ is an isometry, $\wt H$ is self-adjoint on $L^2(\mu)$ and has the
    same eigenvalues as $H$. Note that, for any $f\in\cD(\wt H)$, it follows
    from Green's identity that
    \begin{multline*}
        \langle{f,\wt Hf}\rangle_{L^2(\mu )} = \ang{\pi f,H(\pi f)}_{L^2(\bR^d)}
            = \int |\nab(\pi f)|^2 + \int V(\pi f)^2\\
            = \int |\nab(\pi f)|^2 + \int (\De\pi)\pi f^2
            = \int |\nab(\pi f)|^2 - \int \nab\pi\cdot\nab(\pi f^2).
    \end{multline*}
    Using the product rule, $\nab(gh)=g\nab h+h\nab g$, this simplifies to 
    \begin{multline*}
        \langle{f,\wt Hf}\rangle_{L^2(\mu)}
            = \int (|\nab\pi|^2f^2 + 2f\pi(\nab f \cdot \nab\pi)
                + |\nab f|^2\pi^2 - |\nab\pi|^2f^2
                - \pi(\nab(f^2) \cdot \nab\pi))\\
            = \int (2f\pi(\nab f \cdot \nab\pi)
                + |\nab f|^2\pi^2
                - \pi (\nab(f^2) \cdot \nab\pi))
            = \int |\nab f|^2\pi^2,
    \end{multline*}
    showing that $\wt H$ cannot have a negative eigenvalue. Hence, $\wh\la_0\ge
    0$.  

    By \eqref{superquad3}, we have $\pi\in L^2(\bR^d)$, so that $\pi\in\cD(H)$
    with $H\pi =0$. Hence, since $\wh\la_0$ is nonnegative and has multiplicity
    one, it follows that $\wh\la_0=0$ and $\psi_0$ is proportional to $\pi$.

    Observe that, if $f\in C_c^\infty$, then, using the product rule for the
    Laplacian and the identity $V=\De\pi/\pi$, we have
    \[
        -\wt Hf = -\frac 1\pi\,H(\pi f) = \frac 1\pi(\De(\pi f) - V\pi f)
            = \frac 1\pi(f\De\pi + 2\nab\pi \cdot \nab f + \pi\De f - f\De\pi).
    \]
    Since $2\ep\nab\pi/\pi=-\nab F$, we have $-\ep\wt Hf=\ep\De f-\nab
    F\cdot\nab f$.

    Since $\nab F$ is locally Lipschitz, \eqref{SDE} has a unique solution up to
    an explosion time (see \cite[Theorem V.38]{Pro}). Since
    $\lim_{|x|\rightarrow\infty}F=\infty$ by assumption and
    $\lim_{|x|\rightarrow\infty}AF(x)=\infty$ by Lemma 4.2, it follows that $F$
    is a Liapunov function for $X_\ep$ proving that $X_\ep$ does not explode.  

    By \cite[Remark 2.5]{Ku98}, $A$ is graph separable. Clearly $\cD(A)$ is
    closed under multiplication. Since $\cD(A)$ separates points and $\bR^d$ is
    complete and separable, $\cD(A)$ is separating (see \cite[Theorem
    3.4.5]{EK}).

    If $f\in C_c^\infty$, then
    \[
        \int Af\,d\vpi = -\ep\langle 1,\wt Hf\rangle_{L^2(\vpi)}
            = -\ep\langle\wt H1,f\rangle_{L^2(\vpi)}=0,
    \]
    so that $\vpi$ is a stationary distribution for $A$. For $k\ge 1$, since
    $\vpi_k(dx)=\eta_k(x)\vpi(dx)$, we have
    \[
        \int Af\,d\vpi_k = -\ep\langle\eta_k,\wt Hf\rangle_{L^2(\vpi)}
            = -\ep\langle\wt H\eta_k,f\rangle_{L^2(\vpi)}
            =-\la_k\int f\,d\vpi_k.
    \]
    Also, $\vpi_k(\bR^d)=\ang{\eta_k,1}_{L^2(\vpi)}=0$, since $\eta_k$ and
    $\eta_0=1$ are orthogonal.

    Finally, since $\eta_k=Z^{1/2}\psi_k/\pi$ and $\psi_k$ is locally H\"older
    continuous, it follows that each $\eta_k$ belongs to $C(\bR^d)$, and the
    fact that they are bounded follows from Theorems \ref{T:Davies1} and
    \ref{T:Davies2}.
\end{proof}

\subsection{The coupled process}\label{S:coupproc}

By Proposition \ref{P:opbasics}, the pair $(A,m)$ satisfies Assumption 
\ref{A:eigmeas} with $E=\bR^d$, so we have the following.  

\begin{thm}\label{T:mainRd} 
    Let $A$ be the generator for (\ref{SDE}) where $F$ satisfies Assumption
    \ref{asmF}, and let $(-\la_0,\eta_0),\ldots ,(-\la_m,\eta_m)$ be the
    first $m+1$ eigenvalues and eigenvectors of $A$. Let
    $Q\in\bR^{(m+1)\times (m+1)}$ be the generator of a continuous-time Markov
    chain with state space $E_0=\{0,1,\ldots,m\}$ and eigenvalues
    $\{0,-\la_1,\ldots,-\la_m\}$ and eigenvectors $\xi^{(1)},\ldots,\xi^{(m)}$
    such that $\al_i$ defined by \eqref{alphas} is strictly positive. Let $B$ be
    defined as in Definition \ref{D:Bgen}.  
    
    Let $\wt Y$ be a continuous time Markov chain with generator $Q$ and initial
    distribution $p=(p_0,\ldots,p_m)\in \cP(E_0)$. Then there exists a cadlag
    Markov process $(X,Y)$ with generator $B$ and initial distribution $\nu$
    given by
    \begin{equation}
        \nu(\Ga \times \{i\})
            = p_i\al(i,\Ga), \quad \Ga \in \cB(\bR^d),\label{nudef}
    \end{equation}
    such that $Y$ and $\wt Y$ have the same distribution on $D_{E_0}[0,\infty)$,
    and
    \begin{equation}\label{mainRd}
        P(X(t) \in \Ga \mid Y(t) = j) = \int_\Ga \al_j(x)\,\vpi(dx),
    \end{equation}
    for all $t\ge0$, all $0\le j\le m$, and all $\Ga\in\cB(E)$.
\end{thm}

\begin{rmk}
    That $Q$ with these properties exists can be seen from \cite[Theorem
    1]{Perfect1953}. Remark \ref{R:Qmat} ensures the existence of the
    eigenvectors.
\end{rmk}

\begin{proof}
    Note that under the assumptions of the theorem,
    $(A,m,Q,\xi^{(1)},\ldots,\xi^{(m)})$ satisfies Assumption \ref{A:Qmat}. By
    Proposition \ref{P:opbasics}, the rest of the hypotheses of Theorem
    \ref{T:main} are also satisfied. Consequently, the process $(X,Y)$ exists,
    and by uniqueness of the martingale problem for $B$, $(X,Y)$ is Markov.
\end{proof}

We can now construct the coupling sequences described in the introduction. For
each $\ep>0$, choose a matrix $Q_\ep$ and eigenvectors $\xi_\ep^{
(1)},\ldots,\xi_\ep^{(m)}$ that satisfy the assumptions of Theorem 
\ref{T:mainRd}. If $(X_\ep,Y_\ep)$ is the Markov process described in Theorem
\ref{T:mainRd}, then the family, $\{(X_\ep, Y_\ep):\ep>0\}$, forms a coupling
sequence.

The coupling sequence is determined by the collection, $\{Q_\ep, \xi_\ep^{
(1)},\ldots,\xi_\ep^{(m)}: \ep>0\}$. By making different choices for the
matrices and eigenvectors, we can obtain different coupling sequences. In our
follow-up paper, we will consider the question of existence and uniqueness of a
coupling sequence that satisfies conditions \eqref{tracking-req} and 
\eqref{rate-match-req}.

\section*{Acknowledgments}

This paper was completed while the first author was visiting the University of
California, San Diego with the support of the Charles Lee Powell Foundation. The
hospitality of that institution, particularly that of Professor Ruth Williams,
was greatly appreciated.

\renewcommand{\theequation}{A.\arabic{equation}}
\appendix
\setcounter{equation}{0}

\section{Appendix}

\subsection{The Markov mapping theorem}\label{S:MMT}

Let $E$ be a complete and separable metric space, $\cB(E)$ the $\si$-algebra of
Borel subsets of $E$, and $\cP(E)$ the family of Borel probability measures on
$E$. Let $M(E)$ be the collection of all real-valued, Borel measurable functions
on $E$, and $B(E)\subset M(E)$ the Banach space of bounded functions with
$\|f\|_\infty=\sup_{x\in E}|f(x)|$. Let $\ol C(E)\subset B(E)$ be the subspace
of bounded continuous functions, while $C(E)$ denotes the collection of
continuous, real-valued functions on $E$. A collection of functions $D\subset\ol
C(E)$ is {\it separating} if $\mu,\nu\in\cP(E)$ and $\int f\,d\mu=\int f\,d\nu$
for all $f\in D$ implies $\mu=\nu$.

\begin{condition}\label{mapcond}
    \begin{enumerate}[(i)]
        \item   $B\subset\ol C(E)\times C(E)$ and $\cD(B)$ is closed under
                multiplication and separating.
        \item   There exists $\psi\in C(E)$, $\psi\ge 1$, such that for each
                $f\in\cD(B)$, there exists a constant $c_f$ such that
                \[
                    |Bf(x)| \le c_f\psi(x), \quad x \in E.
                \]
                (We write $Bf$ even though we do not exclude the possibility
                that $B$ is multivalued. In the multivalued case, each element
                of $Bf$ must satisfy the inequality.)
        \item   There exists a countable subset $B_c\subset B$ such that every
                solution of the (local) martingale problem for $B_c$ is a
                solution of the (local) martingale problem for $B$.
        \item   $B_0f\equiv\psi^{-1}Bf$ is a pre-generator, that is, $B_0$ is
                dissipative and there are sequences of functions
                $\mu_n:E\to\cP(E)$ and $\la_n:E\to[0,\infty)$ such that for each
                $(f,g)\in B$,
                \begin{equation}
                    g(x) = \lim_{n\to\infty}
                        \la_n(x)\int_E (f(y) - f(x))\mu_n(x,dy)\label{gencomp}
                \end{equation}
                for each $x\in E$.
    \end{enumerate}
\end{condition}

\begin{rmk}
    Condition \ref{mapcond}(iii) holds if $B_0$ is graph-separable, that is,
    there is a countable subset $B_{0,c}$ of $B_0$ such that $B_0$ is a subset
    of the bounded, pointwise closure of $B_{0,c}$.

    An operator is a pre-generator if for each $x \in E$, there exists a
    solution of the martingale problem for $(B,\de_x)$.
\end{rmk}

For a measurable $E_0$-valued process $Y$, where $E_0$ is a complete and
separable metric space, let
\[
    \wh\cF^Y_t = \text{completion of }\si\left({
            \int_0^r g(Y(s))\,ds: r \le t, g \in B(E_0)
        }\right) \vee \si(Y(0)).\label{obs}
\]

\begin{thm}\label{T:MMT}
    Let $(S,d)$ and $(E_0,d_0)$ be complete, separable metric spaces. Let $B$
    satisfy Condition \ref{mapcond}. Let $\ga :S\to E_0$ be measurable, and let
    $\wt\al$ be a transition function from $E_0$ into $S$ (that is,
    $\wt\al:E_0\times\cB(S)\to\bR$ satisfies $\wt\al(y,\cdot)\in\cP(S)$ for all
    $y\in E_0$ and $\wt\al(\cdot,\Ga)\in B(E_0)$ for all $\Ga\in\cB(S)$)
    satisfying $\int h\circ\ga(z)\,\wt\al(y,dz)=h(y)$, $y\in E_0$, $h\in
    B(E_0)$, that is, $\wt\al(y,\ga^{-1}(y))=1$. Assume that
    $\wt\psi(y)\equiv\int_S\psi(z)\wt\al(y,dz)<\infty$ for each $y\in E_0$ and
    define
    \[
        C = \left\{{\left({
                \int_S f(z)\wt\al(\cdot,dz), \int_S Bf(z)\wt\al(\cdot,dz)
            }\right): f \in \cD(B)}\right\}.
    \]
    Let $\mu\in\cP(E_0)$ and define $\nu=\int\wt\al(y,\cdot)\,\mu(dy)$.
    \begin{enumerate}[a)] 
        \item   If $\wt Y$ satisfies $\int_0^tE[\wt\psi(\wt Y(s))]\,ds<\infty$
                a.s.~for all $t>0$ and $\wt Y$ is a solution of the martingale
                problem for $(C,\mu)$, then there exists a solution $Z$ of the
                martingale problem for $(B,\nu)$ such that $\wt Y$ has the same
                distribution on $M_{E_0}[0,\infty)$ as $Y=\ga\circ Z$. If $Y$
                and $\wt Y$ are cadlag, then $Y$ and $\wt Y$ have the same
                distribution on $D_{E_0}[0,\infty)$.
        \item   Let ${\bf T}^Y=\{t:Y(t)\text{ is }\wh\cF^Y_t\text{
                measurable}\}$ (which holds for Lebesgue-almost every $t$). Then
                for $t \in {\bf T}^Y$,
                \[
                    P(Z(t) \in \Ga \mid \wh\cF_t^Y)
                        = \wt\al(Y(t),\Ga), \quad \Ga \in \cB(S).
                \]
        \item   If, in addition, uniqueness holds for the martingale problem for
                $(B,\nu)$, then uniqueness holds for the
                $M_{E_0}[0,\infty)$-martingale problem for $(C,\mu)$. If $\wt Y$
                has sample paths in $D_{E_0}[0,\infty)$, then uniqueness holds
                for the $D_{E_0}[0,\infty)$-martingale problem for $(C,\mu)$.
        \item   If uniqueness holds for the martingale problem for $(B,\nu)$,
                then $Y$ restricted to ${\bf T}^Y$ is a Markov process.
    \end{enumerate}
\end{thm}

\begin{rmk}
    If $Y$ is cadlag with no fixed points of discontinuity (that is $Y(t) =
    Y(t-)$ a.s.~for all $t$), then $\wh\cF^Y_t=\cF_t^Y$ for all $t$.
\end{rmk}

\begin{rmk}
    The main precursor of this Markov mapping theorem is \cite[Corollary
    3.5]{Ku98}. The result stated here is a special case of Corollary 3.3 of
    \cite{KN11}.
\end{rmk}


\end{document}